\documentclass[11pt]{amsart}

\usepackage{amsmath}
\usepackage{amsfonts}
\usepackage{amsthm}
\usepackage{amssymb}
\usepackage{setspace}
\usepackage{verbatim}

\numberwithin{equation}{section}

\theoremstyle{plain}
\newtheorem{theorem}[equation]{Theorem}

\newtheorem{corollary}[equation]{Corollary}
\newtheorem{conjecture}[equation]{Conjecture}

\newtheorem{proposition}[equation]{Proposition}

\theoremstyle{definition}

\newtheorem{remark}[equation]{Remark}

\newcommand{\Z}{\mathbb{Z}}
\newcommand{\N}{\mathbb{N}}
\newcommand{\Q}{\mathbb{Q}}
\newcommand{\G}{\mathcal{G}}

\usepackage[colorlinks,pdftex, bookmarks=false]{hyperref}

\title[On conjectures and problems of Ruzsa]{On conjectures and problems of Ruzsa concerning difference graphs of $S$-units}

\author{Ante \'Custi\'c}
\address{
Institut f\"ur Optimierung und Diskrete Mathematik (Math B)\\
Technische Universit\"at Graz\\ Steyrergasse 30/II\\
8010 Graz, Austria
}
\email{custic@math.tugraz.at}

\author{Lajos Hajdu}
\address{
Institute of Mathematics, University of Debrecen\\
P.O.\@ Box 12, H-4010 Debrecen, Hungary
}
\email{hajdul@science.unideb.hu}

\author{Dijana Kreso}
\address{
Institut f\"ur Analysis und Computational Number Theory (Math A)\\
Technische Universit\"at Graz\\ Steyrergasse 30/II\\
8010 Graz, Austria
}
\email{kreso@math.tugraz.at}

\author{Robert Tijdeman}
\address{
Mathematical Institute, Leiden University, 2300 RA Leiden, P.O.\@
Box 9512, The Netherlands
}
\email{tijdeman@math.leidenuniv.nl}

\date{}

\begin{document}

\begin{abstract}

Given a finite nonempty set of primes $S$, we build a graph $\mathcal{G}$ with vertex set $\Q$ by connecting $x, y\in \Q$ if the prime divisors of both the numerator and denominator of $x-y$ are from $S$. In this paper we resolve two conjectures posed by Ruzsa concerning the possible sizes of induced nondegenerate cycles of $\mathcal{G}$, and also a problem of Ruzsa concerning the existence of subgraphs of $\G$ which are not induced subgraphs.
\end{abstract}


\maketitle

%################################################################################
%################################################################################
%################################################################################

\section{Introduction and main results}
Let $S$ be a finite nonempty set of primes, $\Z_S$ be the ring of those rationals whose denominators (when written in lowest terms) are not divisible by primes outside $S$, and $\Z_S^*$ the multiplicative group of invertible elements ($S$-units) in $\Z_S$.
We build a graph $\mathcal{G}$ with vertex set $\Q$ by connecting $x, y\in \Q$ if $x-y\in \Z_S^*$.  Then $\mathcal{G}$ is said to be the {\it $S$-unit graph}. Graphs of this type were introduced by Gy\H{o}ry~\cite{G82}, and have been subsequently intensively studied and applied to various Diophantine problems, see \cite{G08} and \cite{ght1}, and references therein. 

With the aim of understanding the structure of the $S$-unit graph $\mathcal{G}$, Ruzsa~\cite{R11} studied its possible subgraphs. We say that distinct $a_1, a_2,\ldots, a_n\in \Q$, $n\geq 3$, form an {\it induced cycle} $a_1\to a_2\to \cdots \to a_n\to a_1$ of $\mathcal{G}$ if $a_j-a_i$ is an $S$-unit if and only if either $j\in \{i+1, i-1\}$ or $\{i, j\}=\{1, n\}$. Given any cycle $a_1\to a_2\to \cdots \to a_n\to a_1$ of $\mathcal{G}$, let $u_i=a_{i+1}-a_i$ for $i=1, 2, \ldots, n-1$, and $u_n=a_1-a_n$. Then clearly $u_1+u_2+\cdots+u_n=0$. Note that the cycle $a_1\to a_2\to \cdots \to a_n\to a_1$ is an induced cycle of $\mathcal{G}$ if and only if
\[
u_{i}+u_{i+1}+\cdots+u_{j}\notin \Z_S^* \ \textnormal{for all}\  1\leq i<j\leq n, \  (i, j)\neq (1,n-1), (2, n).
\]
In \cite{R11}, Ruzsa proved that if $2\in S$, then there exist induced cycles of $\mathcal{G}$ of every length, and if $2\notin S$, then there exist induced cycles of $\mathcal{G}$ of every even length and none of odd length.

Ruzsa further studied nondegenerate induced cycles. A cycle $a_1\to a_2\to \cdots \to a_n\to a_1$ of $\mathcal{G}$ is said to be {\it degenerate} if there exists a proper zero subsum of $u_1+u_2+\cdots+u_n$, i.e.\@ if there exist $l, i_1, i_2, \ldots, i_l\in \N$ satisfying $0<l<n$ and $1\leq i_1<\cdots<i_l\leq n$ such that $u_{i_1}+u_{i_2}+\cdots+u_{i_l}=0$.
If no such proper zero subsum exists, then the cycle is said to be {\it nondegenerate}. Ruzsa proved that if $2\in S$, then there are nondegenerate induced cycles of $\mathcal{G}$ of every length, and if $2\notin S$, but $3\in S$, then there are nondegenerate induced cycles of $\mathcal{G}$ of every even length and none of odd length. Furthermore, Ruzsa posed the following conjecture (which is Conjecture 3.4~in~\cite{R11}).

\begin{conjecture}\label{conj1}
Let $S$ be a finite nonempty set of primes. Then there are nondegenerate induced cycles of the $S$-unit graph $\mathcal{G}$ of every sufficiently large even length.
\end{conjecture}

We show that this conjecture does not hold by proving the following theorem.

\begin{theorem}\label{thm1}
If $S=\{p\}$ for some prime number $p$, then there are nondegenerate induced cycles of the $S$-unit graph $\mathcal{G}$ of length $n$ if and only if $n\equiv 2 \pmod {p-1}$.
\end{theorem}

%If $S=\{2\}$ or $S=\{3\}$, then Theorem~\ref{thm1} follows from the aforementioned Ruzsa's results. 

Ruzsa further posed another conjecture. For a given finite nonempty set of primes $S$ let
\begin{equation}\label{defm}
m=\gcd \{p-1 : p\in S\}.
\end{equation}
Note that if there exist positive $S$-units $u_1, u_2, \ldots, u_k$ such that $u_1+u_2+\cdots+u_k=1$, by multiplying the expression by the least common multiple of the denominators of $u_i$'s and by considering the congruence relation modulo $m$, we get that $k\equiv 1 \pmod{m}$, because $p^{d}\equiv 1 \pmod{m}$ for every $p\in S$ and every $d\in \N\cup\{0\}$. Ruzsa conjectured that for every sufficiently large $k$ such that $k\equiv 1 \pmod{m}$ there exist positive $S$-units $u_1, u_2, \ldots, u_k$ such that $u_1+u_2+\cdots+u_k=1$ and
\[
u_i+u_{i+1}+\cdots+u_j\notin \Z_S^* \quad \textnormal{for all}\  1\leq i<j\leq k, \ (i, j)\neq (1, k).
\]
We show that this conjecture, which is Conjecture 3.5 in \cite{R11} (with a misprint $m\mid k$ in place of $k\equiv 1 \pmod{m}$), is true, i.e.\@ the following theorem holds. 

\begin{theorem}\label{thm2}
Let $S$ be a finite nonempty set of primes and let $m$ be as in \eqref{defm}. Then for every sufficiently large $k$ with $k\equiv 1 \pmod{m}$ there exist positive $S$-units $u_1, u_2, \ldots, u_k$ such that
\begin{equation}\label{sum}
u_1+u_2+\cdots+u_k=1
\end{equation}
and 
\begin{equation}\label{unit}
u_i+u_{i+1}+\cdots+u_j\notin \Z_S^* , \quad \textnormal{for all}\  1\leq i<j\leq k, \ (i, j)\neq (1, k).
\end{equation}
\end{theorem}

%If $2\in S$ or $3\in S$, then Theorem~\ref{thm2} follows from the aforementioned Ruzsa's results. 

Let $m$ be as in \eqref{defm} and let $n\in \N$ be such that $n\equiv 2 \pmod{m}$. 
Take arbitrary $a\in \Q$ and let $k=n-1$. Hence $k\equiv 1 \pmod{m}$. Assume that $k$ is sufficiently large, so that by Theorem~\ref{thm2} there exist positive $S$-units $u_1, u_2, \ldots, u_k$ such that \eqref{sum} and \eqref{unit} hold. Consider the following cycle of the $S$-unit graph $\mathcal{G}$:
\[
a\to a+u_1\to a+u_1+u_2\to \cdots \to a+u_1+u_2+\cdots +u_k\to a.
\]
Note that this is indeed a cycle of $\mathcal {G}$ of length $n$ because \eqref{sum} holds. It is an induced cycle of $\mathcal{G}$ because \eqref{unit} holds. Furthermore, it is nondegenerate because $u_i$'s are all positive, so no proper zero subsum of $u_1+u_2+\cdots+u_n-1$
exists. Therefore, Theorem~\ref{thm2} has the following corollary.

\begin{corollary}
Let $S$ be a finite nonempty set of primes and let $m$ be as in \eqref{defm}. Then for every sufficiently large $n$ with $n\equiv 2 \pmod{m}$ there exists a nondegenerate induced cycle of the $S$-unit graph $\mathcal{G}$ of length $n$.
\end{corollary}

Finally, Ruzsa proposed the following problem (which is Problem 3.6~in~\cite{R11}) concerning general subgraphs of the $S$-unit graph $\G$.  Recall that a graph $G=(V, E)$ with vertex set $V$ and edge set $E$ is a subgraph of $\G$ if to each $i\in V$ we can assign a rational {\it value} $q_i$ so that $q_i\neq q_j$ if $i\neq j$, and that if there exists an edge $(i, j)\in E$ between $i$ and $j$ in $G$, then $q_i-q_j$ is an $S$-unit. A subgraph $G=(V, E)$ of $\G$ is said to be {\it induced} subgraph of $\G$ if to each $i\in V$ we can assign a rational value $q_i$ so that $q_i\neq q_j$ if $i\neq j$, and that there exists an edge between $i$ and $j$ in $G$ if and only if $q_i-q_j$ is an $S$-unit. Ruzsa asked if there exists a finite subgraph of $\G$ which is not induced subgraph of $\G$, and further noted that he expects a positive answer.  His motivation to formulate this problem was that in his constructions the main difficulty arose from the requirement on cycles to be induced. We give an affirmative answer to Ruzsa's question by proving the following theorem.

\begin{theorem} \label{thm3}
Given a finite nonempty set $S$ of primes, there exists a subgraph of the $S$-unit graph $\G$, which is not an induced subgraph of $\G$.
\end{theorem}

In Section~\ref{SecProofs} we prove Theorem~\ref{thm1}, Theorem~\ref{thm2} and Theorem~\ref{thm3}, and in Section~\ref{SecRemarks} we make some further remarks on Conjecture~\ref{conj1}.

%################################################################################
%################################################################################

\section{Proofs of the theorems} \label{SecProofs}
Throughout this section $S$ denotes a fixed finite nonempty set of primes and $m=\gcd \{p-1 : p\in S\}$. We first prove Theorem~\ref{thm2}. Call $k\in \N$ {\it $S^+$-good} if there exist positive $S$-units $u_1, u_2, \ldots, u_k$ such  that
\begin{equation}\label{S-units}
u_1+u_2+\cdots+u_k=1.
\end{equation}
Note that $p$ is $S^+$-good for all $p\in S$; this follows by setting $u_i=1/p$ in \eqref{S-units} for all $i=1, 2, \ldots, k$. We will use the fact that the number of solutions in positive $S$-units $u_1, u_2, \ldots, u_k$ of equation \eqref{S-units} for a fixed $k\in \N$ is finite. That follows from the following deep result of Van der Poorten and Schlickewei~\cite{PS82} and, independently, Evertse~\cite{E84}. 

\begin{proposition}\label{Evertse}
Let $S$ be a finite nonempty set of primes, $k\in \N$. Then the number of solutions of the equation
\[
u_1+u_2+\cdots+u_k=1
\]
in $S$-units $u_1, u_2, \ldots, u_k$ such that
\[
u_{i_1}+u_{i_2}+\cdots+u_{i_l}\neq 0
\]
for all $l, i_1, i_2, \ldots, i_l\in \N$ satisfying $0<l<k$ and $1\leq i_1<\cdots<i_l\leq k$, is bounded.
\end{proposition}

To prove Theorem~\ref{thm2} we further need the following proposition.

\begin{proposition}\label{Lemma}
Let $S$ be a finite nonempty set of primes and let $m$ be as in \eqref{defm}. For every sufficiently large $k$ such that $k\equiv 1 \pmod{m}$ there exist positive $S$-units $u_1, u_2, \ldots, u_k$ such that
\[
u_1+u_2+\cdots+u_k=1.
\]
\end{proposition}

\begin{proof}
We want to prove that any sufficiently large $k\in \N$ such that $k\equiv 1 \pmod{m}$ is $S^+$-good.
Recall that $p$ is $S^+$-good for all $p\in S$. If $\#S\geq 2$ and $p, q\in S$, note that
\[
\underbrace{\frac{1}{q}+ \cdots +\frac{1}{q}}_{q-1}+\underbrace{\frac{1}{pq}+\cdots+\frac{1}{pq}}_{p}=1,
\]
hence $k=q-1+p=(p-1)+(q-1)+1$ is $S^+$-good as well. Further, note that for any $p\in S$ we can replace any $S$-unit $u>0$ by a sum of $p$ positive $S$-units by using
\[
u=\underbrace{\frac{u}{p}+\cdots+\frac{u}{p}}_{p}.
\]
Writing $S=\{p_1, p_2, \ldots, p_d\}$, it follows that any $k$ of type
\begin{equation}\label{good}
a_1(p_1-1)+a_2(p_2-1)+\cdots+a_d(p_d-1)+1,
\end{equation}
with $a_i\in \N\cup \{0\}$, $i=1, 2, \ldots, d$, is $S^+$-good as well, because we may lengthen the sum by $p_i-1$ as many times as we like for all $i=1, 2, \ldots, n$. Since $m=\gcd\{p_1-1, p_2-1, \ldots, p_d-1\}$, there exist integers $b_i\in \Z$, $i=1, 2, \ldots, d$ such that
\begin{equation}\label{m}
m=b_1(p_1-1)+b_2(p_2-1)+\cdots+b_d(p_d-1).
\end{equation}
Let $M:=(p_1-1)/m$ and let $k_0$ be as follows
\begin{equation}\label{k_min}
k_0=M|b_1|(p_1-1)+\cdots +M|b_d|(p_d-1)+1.
\end{equation}
Then from \eqref{good} it follows that $k_0$ is $S^+$-good. Further note that for every $j\in\{1,\ldots,M-1\}$, $k_0+jm$ is of type \eqref{good}, and is hence $S^+$-good; this follows from adding \eqref{m} to \eqref{k_min} $j$ times. Finally we show that every $k\geq k_0$ such that $k \equiv 1 \pmod{m}$ is of type \eqref{good}, and is hence $S^+$-good.

Indeed, write $k-k_0=q(p_1-1)+r_1$ where $q, r_1\in \N\cup\{0\}$ and $r_1<p_1-1$. Then $r_1=rm$ for some $r\in \{0, 1,\ldots,M-1\}$. We have that $k_0+rm$ is $S^+$-good, as shown in the previous step, and hence $k=k_0+rm+q(p_1-1)$ is $S^+$-good as well, because it is of type \eqref{good}.
\end{proof}

Next we prove Theorem~\ref{thm2} by using Proposition~\ref{Lemma} and by following the approach of Ruzsa from the proof of Theorem 3.3\ in \cite{R11}.

\begin{proof}[Proof of Theorem~\ref{thm2}]
From Proposition~\ref{Lemma} it follows that every sufficiently large $k\in \N$ such that $k\equiv 1 \pmod{m}$ is $S^+$-good. So let $k\equiv 1 \pmod{m}$ be sufficiently large and let $u_1, u_2, \ldots, u_{k}$ be positive $S$-units such that
\begin{equation}\label{sumproof}
u_1+u_2+\cdots+u_{k}=1.
\end{equation}
Recall that there are only finitely many solutions $(u_1, u_2, \ldots, u_k) $ in positive $S$-units of equation \eqref{sumproof} for fixed $k$, see Proposition~\ref{Evertse}. From these solutions select the lexicographically last and denote it by $(u_1', u_2', \ldots, u_{k}')$. Assume that there exists $I\subseteq \{1, 2, \ldots, k\}$ with $\#I\geq 2$ such that
\begin{equation}\label{unitsum}
v:=\sum_{i\in I} u_i'\ \textnormal{is an $S$-unit}.
\end{equation}
We will show that then necessarily $I=\{j, j+1, \ldots, k\}$ for some $j$ with $1\leq j\leq k-1$. Note that from \eqref{unitsum} it follows that $\# I\equiv 1 \pmod{m}$. Let $j'$ be the minimal element of $I$. Assume to the contrary that there exists $l\notin I$ such that $j'<l\leq k$. Then we can find a solution of \eqref{sumproof} which is lexicographically later than $(u_1', u_2', \ldots, u_{k}')$ by the following transformations. Replace $u_{j'}'$ by $v$ and delete all $u_i'$, $i\in I$, $i\neq j'$; then replace $u_l$ by $\# I$ $S$-units whose sum is $u_l$, which can be obtained by multiplying \eqref{unitsum} by the $S$-unit $u_lv^{-1}$. This is a solution of equation \eqref{sumproof} which is lexicographically later since we increased $u_{j'}'$, and left $u_1', u_2', \ldots, u_{j'-1}'$ unchanged. However, this is a contradiction.

Hence, there exist positive $S$-units $u_1', u_2', \ldots, u_{k}'$ satisfying $u_1'+u_2'+\cdots+u_{k}'=1$ and such that if
\begin{equation}\label{sumI}
\sum_{i\in I} u_i'\in \Z_S^*, \quad \textnormal{for some}\ \# I\geq 2,
\end{equation}
then $I=\{j, j+1, \ldots, k\}$ for some $1\leq j\leq k-1$. We may choose positive $S$-units $u_1^*, u_2^*, \ldots, u_{k}^*$ that satisfy $u_1^*+u_2^*+\cdots+u_{k}^*=1$ and
\[
u_{i}^*+u_{i+1}^*+\cdots+u_{j}^*\notin \Z_S^*, \quad 1\leq i<j\leq k, \quad (i, j)\neq (1, k),
\]
by setting $u^*_i=u'_{i-1}$ for $i=2,\ldots,k$ and $u^*_1=u'_{k}$. Indeed, if $u_{i}^*+u_{i+1}^*+\cdots+u_{j}^*$, with  $1\leq i<j\leq k$, is an $S$-unit, then it follows from \eqref{sumI} that both $u_{k-1}'$ and $u_k'$ appear in the sum, i.e.\ both $u_k^*$ and $u_{1}^*$ appear in the sum, hence $i=1$ and $j=k$.
\end{proof}

We now prove Theorem~\ref{thm1}.

\begin{proof}[Proof of Theorem~\ref{thm1}]  Let $S=\{p\}$ for some prime number $p$. We first prove that if $n\equiv 2 \pmod{p-1}$, then there exist nondegenerate induced cycles of $\G$ of length $n$. Let $k=n-1$. Then $k\equiv 1 \pmod{p-1}$ and by Theorem~\ref{thm2} there exist positive $S$-units $u_1, u_2, \ldots, u_{k}$ such that $u_1+u_2+\cdots+u_{k}=1$ and that
\begin{equation}\label{noedges}
u_{i}+u_{i+1}+\cdots+u_{j}\notin \Z_S^*, \quad \textnormal{for all} \ 1\leq i<j\leq k, \ (i, j)\neq (1, k).
\end{equation}
Namely, since $S=\{p\}$ it follows that $m=p-1$, where $m$ is defined as in \eqref{defm}.
In fact, we can give such $S$-units explicitly. Indeed, if $d=(k-1)/(p-1)$, then
\begin{equation}\label{sum1}
\frac{1}{p^d}+\underbrace{\frac{1}{p}+\cdots+\frac{1}{p}}_{p-1}+\underbrace{\frac{1}{p^2}+ \cdots +\frac{1}{p^2}}_{p-1}+ \cdots +\underbrace{\frac{1}{p^d}+\cdots+\frac{1}{p^d}}_{p-1} =1.
\end{equation}
One easily checks that the condition \eqref{noedges} is also satisfied. Note that for arbitrary $a\in \Q$, the cycle $a\to a_1+u_1\to a+u_1+u_2\to \cdots \to a+u_1+\cdots+u_{k}\to a$ is a nondegenerate induced cycle of length $n$. It is indeed an induced cycle because condition \eqref{noedges} holds. It is nondegenerate because all the $u_i$'s are positive.

Next we prove that the condition $n\equiv 2 \pmod {p-1}$ is  necessary. Assume that there exists a nondegenerate induced cycle
\begin{equation}\label{cycle}
a_1\to a_2\to \cdots \to a_n\to a_1
\end{equation}
of $\G$ of length $n$. Let $u_i=a_{i+1}-a_i$, $i=1, 2, \ldots, n-1$ and $u_n=a_1-a_n$. Then 
\begin{equation}\label{nulsuma}
u_1+u_2+\cdots+u_n=0.
\end{equation}
Without loss of generality we may assume that $u_i\in\Z$ for all $i=1, 2, \ldots, n$, as we can multiply them by the least common multiple of the denominators of $u_i$'s. In doing so the corresponding cycle remains to be nondegenerate induced cycle of $\G$. Then $u_i\in \{p^{k_i}, -p^{k_i}\}$ for some $k_i\in \N\cup\{0\}$. By assumption no proper zero subsum of \eqref{nulsuma} exists. In what follows we show that among all $u_i's$ only one is negative or only one is positive.

Without loss of generality assume that $u_1$ is the smallest in absolute value among all the $u_i$'s.  We may further assume $u_1=1$, since otherwise we may divide \eqref{nulsuma} by $u_1$. Note that then $u_i\neq -1$ for all $i\geq 2$, since otherwise $u_1+u_i=0$, which would lead to degeneracy of the cycle \eqref{cycle}. Let $b_1$ be the total number of occurrences of $1$'s in the sum. Since $u_1+\cdots+u_n\equiv 0 \pmod{p}$ it follows that $b_1\equiv 0\pmod{p}$. Hence, we can group the $u_i$'s with value 1 into \emph{blocks} of size $p$. By nondegeneracy, either $u_i\neq-p$ for all $i=1, 2, \ldots, n$, or there is exactly one negative $u_i$ (and it equals $-p$) in which case we are done. 
 %and $n=p+1=p-1+2$, since otherwise a zero subsum would appear. 
Now, let $b_p$ be the total number of occurrences of $p$'s in the sum plus the number of $p$-blocks of 1's. Then, since $u_1+\cdots+u_n\equiv 0 \pmod{p^2}$ it follows that $b_p\equiv 0\pmod{p}$. Then by nondegeneracy, 
either $u_i\neq -p^2$ for all $i=1, 2, \ldots, n$ or there is exactly one negative $u_i$ (and it equals $-p^2$) in which case we are done. 
% is one of $2p, 3p-1, \ldots, p^2+1$, i.e.\@ $b_p=t(p-1)+2$ for some $t=2, \ldots, p+1$. 
Continuing this inductive reasoning we get that only one $u_i$ is negative among all $u_i$'s, say $u_n$, hence $p^{k_1}+p^{k_2}+\cdots+p^{k_{n-1}}=p^{k_n}$. Since $p^{k_i}\equiv 1\pmod{p-1}$ for all $i=1, 2, \ldots, n$, it follows that $n-1\equiv 1\pmod{p-1}$, hence $n\equiv 2 \pmod{p-1}$.
\end{proof}

\begin{remark}\label{notind}
Note that in the proof of Theorem~\ref{thm1} it is shown that if $\#S=1$, and there exists a nondegenerate (not necessarily induced) cycle of $\G$ of length $n$, then $n\equiv 2 \pmod{p-1}$. 
\end{remark}

Before proving Theorem~\ref{thm3} we state an auxiliary result from the theory of Diophantine equations.
Consider the $S$-unit equation
\begin{equation}
\label{eq5.1}
ax+by=1\ \ \ \text{in}\ x,y\in \Z_S^*,
\end{equation}
where $a$ and $b$ are nonzero rationals. The following result is due to Evertse~\cite{ev}.

\begin{proposition}\label{ev}
The number of solutions of \eqref{eq5.1} is at most
\begin{equation}
\label{neweq5}
3\cdot 7^{2|S|+3}.
\end{equation}
\end{proposition}

The finiteness of the number of solutions of the equation \eqref{eq5.1} can easily be derived from a paper of Mahler, ~\cite{mah}, an effective version of it from a paper of Coates, ~\cite{coat}.

Next we introduce two more notions concerning subgraphs of $\G$. Let graph $G=(V, E)$ be a subgraph of $\G$. Then to each $i\in V$ we can assign a rational value $q_i$, so that $q_i\neq q_j$ if $i\neq j$, and so that if there exists an edge $(i, j)\in E$, then $q_i-q_j$ is an $S$-unit. Then the graph with vertex set $\{q_i:i\in V\}\subseteq \Q$ and edge set defined by connecting $q_i$ and $q_j$ if $(i, j)\in E$, is a {\it representation} of $G$ in $\G$. If $G$ is an induced subgraph of $\G$, then by definition there exists a representation of $G$ in $\G$ with vertex set $\{q_i:i\in V\}\subseteq \Q$ and edges between vertices $q_i$ and $q_j$ if and only if $q_i-q_j$ is an $S$-unit. We say that such a representation of $G$ is an {\it induced representation} of $G$ in $\G$.

\begin{remark}\label{remark}
Note that the set of all representations of subgraph $G$ of $\G$ consists of equivalence classes, where inside each equivalence class any representation of $G$ can be obtained from any other by adding some fixed rational to the values of the vertices and then multiplying these new values by some fixed $S$-unit. From an equivalence class we can therefore choose a unique representative by fixing the value of one vertex to be $0$ and the value of some other vertex, which is connected to this vertex by an edge, to be $1$. 
\end{remark}

\begin{proof}[Proof of Theorem \ref{thm3}]
Let $S$ be a fixed finite nonempty set of primes and $\mathcal{G}$ the $S$-unit graph. By Proposition~\ref{ev} the equation 
\begin{equation} \label{eqev}
1+x=y+z
\end{equation}
in $S$-units $x,y,z$ has only finitely many nondegenerate solutions (a solution $(x,y,z)$ is degenerate if it is of type $(x,1,x)$ or $(x, x,1)$ or $(-1, y, -y)$, and nondegenerate otherwise). We distinguish two cases: 
\begin{enumerate}\label{cases}
\item[(1)] Equation (\ref{eqev}) has no nondegenerate solutions, \label{item1}
\item[(2)] Equation (\ref{eqev}) has nondegenerate solutions.  \label{item2}
\end{enumerate}
First we resolve Case~(1). Choose $S$-units $a,b$ such that $c_0+c_1 a+c_2 b\in \Z_S^*$ with $c_0,c_1,c_2 \in \{-1,0,1\}$, implies 
\[
(c_0, c_1, c_2)\in \{(1, 0, 0), (-1, 0, 0), (0,1, 0), (0, -1, 0), (0, 0, 1), (0, 0, -1)\}.
\]
This is possible since, by Proposition~\ref{ev}, for given $c_0,c_1,c_2\in \Q$ with $c_0c_1c_2 \not= 0$, the equation $c_1x+c_2y=c_0$ has only finitely many solutions in $S$-units $x,y$. Consider graph $G$ with the vertex set
\[
V=\left \{v_0, v_1, v_a, v_b, v_{1+a}, v_{1+b}, v_{a+b}, v_{1+a+b}\right\},
\]
and edge set defined in the following way: connect $v_i, v_j\in V$ by an edge if their subscripts $i$ and $j$ differ by $1,a$ or $b$. Note that the choice of $a$ and $b$ implies that no other difference of subscripts of vertices in $V$ is an $S$-unit. It follows that $G$ is an induced subgraph of $\G$. Indeed, we may assign the value of vertex $v_i\in V$ to be $i$, and then there is an edge between $v_i$ and $v_j$ if and only if $i-j$ is an $S$-unit. Now omit the edge between $v_{1+a}$ and $v_{1+a+b}$. The resulting graph $G^-$ is clearly a subgraph of $\G$. In what follows, we show that $G^-$ is not an induced subgraph of $\G$. 

Suppose the contrary and consider an induced representation of $G^-$ such that the values of $v_0$ and $v_1$ are 0 and $1$, respectively. By Remark~\ref{remark} such a representation exists. Consider the cycle $v_0 \to v_1 \to v_{1+a} \to v_a\to v_0$ in $G^-$. Note that, as we are in Case (1), from the fact that $v_0$, $v_1$ and $v_a$ have distinct values, it follows that the value of $v_a$ has to be some $S$-unit $u_a$ and the value of $v_{1+a}$ has to be $1+u_a$. Analogously it follows that the values of $v_b$ and $v_{1+b}$ have to be $S$-units $u_b$ and $1+u_b$, respectively. Next we consider cycle $v_b\to v_{1+b}\to v_{1+a+b}\to v_{a+b}\to v_b$ in $G^-$. It follows that $v_{a+b}$ has value $u_a+u_b$ and $v_{1+a+b}$ has value $1+u_a+u_b$. Since the difference of the values of vertices $v_{1+a}$ and $v_{1+a+b}$ is $u_b$, i.e.\@ an $S$-unit, and there is no connecting edge between these two vertices, we have a contradiction. Hence graph $G^-$ is not an induced subgraph of $\G$.

Next we consider Case~(2). Let $(x_i,y_i,z_i)$, $i=1,\ldots, k$, be all nondegenerate solutions in $S$-units of \eqref{eqev}. Let 
\[
R_0=\{0, 1, x_i, y_i,z_i, 1+x_i~|~i=1, 2, \ldots, k\}.
\]
Consider graph $G_0$ with vertex set $V_0= \{v_i~|~i\in R_0 \}$ and edge set defined in the following way: connect $v_i, v_j\in V_0$ by an edge if and only if their subscripts $i$ and $j$ differ by an $S$-unit. Choose $S$-unit $a$ such that $a\notin R_0$ and that $a+v-w \in \Z_S^*$ for $v,w \in R_0$ implies $v=w$.
This is possible since, by Proposition~\ref{ev}, for a given $c \not= 0$ the equation $x+y=c$ has only finitely many solutions in $S$-units $x,y$ (take $c=v-w$ for each possible choice of $v,w \in R_0$). Let $R_a = \{a+i~|~i \in R_0\}$. Note that the choice of $a$ implies that $R_0\cap R_a=\emptyset$. 
%Indeed, otherwise there exist $v'$ and $w'$ from $R_0$ such that $a=w'-v'$ and hence $a+v'+0=w'\in \Z_S^*$, but $v'\neq 0$, which is in contradiction with the choice of $a$. 
Consider graph $G_a$ with vertex set $V_a:=\{v_{a+i}~|~i\in R_0\}$ and graph $G$ with the vertex set $V:=V_0\cup V_a$, both with edges between vertices if and only if their subscripts differ by an $S$-unit.  Now we omit from $G$ the edge between $v_{z_1}$ and $v_{a+z_1}$ and denote the resulting subgraph of $\G$ by $G^-$. Note that $G$ is an induced subgraph of $\G$. Further note that $G_0$ and $G_a$ are isomorphic graphs with an isomorphism $f(v_i)=v_{i+a}$, and are subgraphs of $G^-$. We will show that $G^-$ is not an induced subgraph of $\G$. 

Suppose the contrary and consider values for the vertices in $V$ such that the resulting representation of $G^-$ in $\G$ is induced. Without loss of generality we may assume that in this representation of $G^-$ the value of $v_0$ is $0$ and the value of $v_1$ is $1$, see Remark~\ref{remark}. In what follows, we show that the set of values of vertices in $V_0$ in the considered representation of $G^-$ must be $R_0$. Indeed, let $n_i$ denote the number of solutions of 
\begin{equation}\label{defn_i}
1+x_i=y+z
\end{equation}
in $S$-units $y,z$ for $i = 1,2,...,k$. Without loss of generality we may assume $n_1 \geq n_2 \geq \dots \geq n_k$. Suppose $n_1 = n_2 = \cdots = n_l>n_{l+1}$ for some $l\in \{1, 2, \ldots, k\}$. Since for $i\in \{1, 2, \ldots, l\}$ there exists an edge between $v_1$ and $v_{1+x_i}$, it follows that the value of $v_{1+x_i}$ is $1+c$ for some $S$-unit $c$. Since by \eqref{defn_i} there are $n_1$ paths of length $2$ between $v_0$ and $v_{1+x_i}$, it follows that $1+c=y+z$ must have at least $n_1$ solutions. Note that hence $c=x_i$ for some $i\in \{1, 2, \ldots, l\}$. Hence the value of $v_{1+x_i}$ is in the set $\{1+x_1, 1+x_2, \dots, 1+x_l\}$, so the set of values of vertices in $\{v_{1+x_1}, v_{1+x_2}, \dots, v_{1+x_l}\}$ is $\{1+x_1, 1+x_2, \dots, 1+x_l\}$. Therefrom it follows that the set of values of vertices $\{v_{x_i}, v_{y_i}, v_{z_i}~|~i=1, 2, \ldots, l\}$ is $\{x_i, y_i, z_i~|~i=1, 2, \ldots, l\}$. By considering the set of values of vertices $v_{1+x_i}$,  where $i$'s are such that the number of solutions of \eqref{defn_i} equals to $n_{l+1}$, we find that the set of values of these vertices is uniquely determined as well. By proceeding in this way we find by induction that the set of values of vertices of $G_0$ in the considered induced representation of $G^-$ must be $R_0$.

Hence the values of vertices in $V_a$ in the considered representation of $G^-$ are in $\Q\setminus{R_0}$. Since there exists an edge between $v_0$ and $v_a$ in $G^-$, there exists $S$-unit $b$ such that in this representation of $G^-$ the values of $v_0$ and $v_a$ differ by $b$, i.e.\@ the value of $v_a$ is $b$. Note that the representation of $G_0$ with vertex set $R_0$ contains all nondegenerate cycles of length $4$ of $\G$ which contain an edge between $0$ and $1$. It follows that the cycle $v_0 \to v_1 \to v_{1+a} \to v_a\to v_0$ in $G^-$ is degenerate, since the value of $v_a$ is not in $R_0$. Hence the values of $v_a$ and $v_{1+a}$ differ also by 1. Now, recall that $f:V_0\to V_a$ with $f(v_i)=v_{i+a}$ is an isomorphism of graphs $G_0$ and $G_a$, so by the same argument as above (that the set of values of vertices of $G_0$ in the considered representation of $G^-$ is $R_0$), it follows that the set of values of vertices of $G_a$ is $R_0+b=\{i+b~|~i\in R_0\}$. This implies that in $G^-$ the number of edges between $V_0$ and $V_a$ is at least $N:=\#V_0=\#V_a$. This is a contradiction, since by the choice of $a$, the number of edges in $G$ between $V_0$ and $V_a$ is exactly $N$, and hence the number of such edges in $G^-$ is one less. 
\end{proof}

\section{Further remarks}\label{SecRemarks}
Throughout this section as well, $S$ denotes a fixed finite nonempty set of primes, $\G$ denotes the $S$-unit graph, and $m=\gcd \{p-1 : p\in S\}$.

We have proved that the Conjecture~\ref{conj1} does not hold when $\#S=1$, however we believe that the following holds.
\begin{conjecture}\label{NewConj}
Let $S$ be a given finite set of primes with $\#S\geq 2$. There are nondegenerate induced cycles of the $S$-unit graph $\G$ of every sufficiently large even length. 
\end{conjecture}

If $2\in S$ or $3\in S$ the conjecture is true by Ruzsa's results. Moreover, in these cases there are nondegenerate induced cycles of $\mathcal{G}$ of every even length. 

Proving Conjecture~\ref{NewConj} is equivalent to showing that for $S$ with $\#S\geq 2$ and every sufficiently large even $n\in \N$ there exist $S$-units $u_1, u_2, \ldots, u_n$ with zero sum, such that for all $l, i_1, i_2, \ldots, i_k\in \N$ with $0<l<n$ and $1\leq i_1<\cdots<i_l\leq n$ we have
\begin{equation}\label{units2}
\nonumber u_{i_1}+u_{i_2}+\cdots+u_{i_l}\neq 0,
\end{equation}
i.e.\@ the condition on nondegeneracy is satisfied,
and that for all $1\leq i<j\leq n, \  (i, j)\neq (1,n-1), (2, n)$ we have
\[
u_{i}+u_{i+1}+\cdots+u_{j}\notin \Z_S^*,
\]
i.e.\@ the condition on induced cycles is satisfied. It is reasonable to attempt to prove this by splitting the proof into two steps, as it was done in the proof of Theorem~\ref{thm2}. To that end we prove the following proposition which corresponds to Proposition~\ref{Lemma}. To prove the proposition we will use the following well-known fact. For irrational $r$ we have that
\begin{equation}\label{fractional}
\left\{\{nr\}:n\in \N\right\} \ \textnormal{is dense in}\ [0, 1),
\end{equation}
where $\{nr\}$ denotes the fractional part of $nr$. There is a more general result than what is stated above, known as Kronecker's theorem. It can be found in \cite[Chap.\@ 23]{HW}.

\begin{comment}
Indeed, for arbitrary $N\in \N$ by the pigeonhole principle there exist distinct $i, j\in \{ 1, \ldots, N+1\}$, $i<j$, and $k\in \{0, 1\ldots ,N-1\}$ such that 
\begin{equation}
\frac{k}{N}\leq \{ir\}, \{jr\}<\frac{k+1}{N}
\end{equation}
Then $\{(j-i)r\}<1/N$, so every point in $[0, 1)$ is within $1/N$ of the set $\left\{\{n(j-i)r\}: n\in \N\right \}$, and so it follows that $\{\{nr\}:n\in \Z\}$ is dense in $[0, 1)$. 
\end{comment}

\begin{proposition}\label{problem}
Given a finite set of primes $S$ such that $\#S\geq 2$, for every sufficiently large even $n\in \N$ there exist $S$-units $u_1, u_2, \ldots, u_n$ such that 
\begin{equation}\label{sum2}
u_1+u_2+\cdots+u_n=0
\end{equation}
and such that for all $l, i_1, i_2, \ldots, i_k\in \N$ with $0<l<n$ and $1\leq i_1<\cdots<i_l\leq n$ we have
 \begin{equation}
u_{i_1}+u_{i_2}+\cdots+u_{i_l}\neq 0.
\end{equation}
\end{proposition}

\begin{proof}

If $2\in S$, the statement (in fact even Conjecture~\ref{NewConj}) follows from the aforementioned result of Ruzsa. Assume henceforth that $2\notin S$. Let $p,q\in S$ be such that $p>q$, and let $m$ be as in \eqref{defm}. From Theorem~\ref{thm2} 
it follows that there exists $\ell_0\in \N$ such that for every $\ell\in \N$ with $\ell\geq \ell_0$ there exist positive $S$-units $v_1,v_2\dots,v_{\ell m+1}$ with
\begin{equation}\label{nulla}
v_1+v_2+\dots+v_{\ell m+1}=1.
\end{equation}
In what follows we show that for every $r\in \N$ with $1\leq r\leq m/2$ there exist exponents $\alpha_r, \beta_r\in \N$ such that
\begin{equation}\label{egy}
rq^{\beta_r}>p^{\alpha_r}>(r-1)q^{\beta_r}.
\end{equation}
If $r=1$, required $\alpha_1, \beta_1$ clearly exist. Assume henceforth $r>1$, and observe that \eqref{egy} is equivalent to
\begin{equation}
\label{ketto}
{\frac{\log r}{\log q}}>\alpha_r{\frac{\log p}{\log q}}-\beta_r> {\frac{\log (r-1)}{\log q}}.
\end{equation}
Further note that $\log p/\log q$ is irrational, and that clearly
\[
1>{\frac{\log r}{\log q}}>{\frac{\log (r-1)}{\log q}}\geq 0.
\]
Then by \eqref{fractional} there exists $\alpha_r\in \N$ such that 
\begin{align*}
{\frac{\log r}{\log q}}>
\left\{\alpha_r{\frac{\log p}{\log q}}\right\}
>{\frac{\log (r-1)}{\log q}}.
\end{align*}
Let $\beta_r$ be the integer part of $\alpha_r\log p/\log q$. Note that the assumption $p>q$ implies $\beta_ r\in \N$. Hence the assertion \eqref{ketto}, and consequently \eqref{egy}, follows.

Now write
\begin{equation}\label{harom}
q^{\beta_r}+\dots+q^{\beta_r}=p^{\alpha_r}+1+\dots+1,
\end{equation}
where the number of $q^{\beta_r}$'s on the left hand side is $r$, and the number of $1$'s on the right hand side is $r q^{\beta_r}-p^{\alpha_r}$. Note that by $rq^{\beta_r}-p^{\alpha_r}>0$ from \eqref{egy} we get that there is at least one appearance of $1$ on the right hand side. Further note that in total we have
\begin{equation}\label{s_r}
s_r:=r+1+(r q^{\beta_r}-p^{\alpha_r})
\end{equation}
summands in \eqref{harom}. Let $n$ be an even positive integer such that
\begin{equation}\label{defn}
n\geq \ell_0m+\max\limits_{1\leq r\leq m/2}s_r.
\end{equation}
Recall that $p\equiv q\equiv 1\pmod{m}$ and note that from \eqref{s_r} it follows that $s_r\equiv 2r\pmod{m}$. So $n\equiv s_r \pmod{m}$ for some $r\in \{1,\dots,m/2\}$. Hence $n=s_r+\ell m$ for some $\ell\in \N$, where $r\in \{1, \ldots, m/2\}$ and $\ell\geq \ell_0$ by \eqref{defn}. By \eqref{nulla} and \eqref{harom} there exist positive $S$-units $v_1, \ldots, v_{\ell m+1}$ such that
\begin{equation}\label{negy}
p^{\alpha_r}+1+\dots+1+v_1+\dots+ v_{\ell m+1}-q^{\beta_r}-\dots -q^{\beta_r}=0,
\end{equation}
where the number of $1$'s is $rq^{\beta_r}-p^{\alpha_r}-1$, the number of $q^{\beta_r}$'s is $r$, and the $v_i$'s have sum $1$. Note that all the summands in \eqref{negy} are $S$-units, and that their number is
\[
1+(rq^{\beta_r}-p^{\alpha_r}-1)+(\ell m+1)+r=n.
\]
Suppose that in \eqref{negy} we have a proper zero subsum. If $p^{\alpha_r}$ occurs in this subsum, then by $p^{\alpha_r}>(r-1)q^{\beta_r}$ from \eqref{egy} we get that all the $q^{\beta_r}$'s are involved in this zero subsum. This is clearly possible only if the subsum involves all the summands in \eqref{negy}, which is a contradiction. On the other hand, if $p^{\alpha_r}$ does not occur in the subsum, then by $q^{\beta_r}>rq^{\beta_r}-p^{\alpha_r}$ from \eqref{egy} we get a contradiction again. Thus \eqref{negy} has no proper zero subsum, which concludes the proof.
\end{proof}

Note that by Remark~\ref{notind} it follows that the condition $\#S\geq 2$ in Proposition~\ref{problem} is necessary. Now, let $S$ be a finite set of primes with $\# S\geq 2$, and let $n_0$ be a sufficiently large even integer so that there exists a nondegenerate cycle of $\G$ of length $n_0$. Such $n_0$ exists by Proposition~\ref{problem}. To prove Conjecture~\ref{NewConj} one must show that there exists some $n_0'\in \N$ with $n_0'\geq n_0$, such that for every even integer $n$ with $n\geq n_0'$ among all nondegenerate solutions in $S$-units of $u_1+\dots+u_n=0$ there exists a solution $(u_1^*, u_2^*, \ldots, u_n^*)$ such that 
\[
u_{i}^*+u_{i+1}^*+\cdots+u_{j}^*\notin \Z_S^* \ \textnormal{for all}\  1\leq i<j\leq n, \  (i, j)\neq (1,n-1), (2, n).
\]
This remains an open problem.

\subsection*{Acknowledgements}
The first and the third author were supported by the Austrian Science Fund (FWF) W1230-N13 and NAWI Graz. 
The second author was supported in part by the OTKA grants K100339, NK101680 and by the T\'AMOP-4.2.2.C-11/1/KONV-2012-0001 project. The project has been supported by the European Union, co-financed by the European Social Fund.

%################################################################################
%################################################################################
%################################################################################
%################################################################################
%################################################################################

\end{document}